\documentclass[12pt]{amsart}
\usepackage{amsfonts}
\usepackage{latexsym,amsmath,amsthm,amssymb,amsfonts,mathrsfs}
\usepackage{graphicx}
\usepackage[usenames, dvipsnames]{color}
\usepackage{tikz}
\usepackage{float}
\usepackage{ulem}

\numberwithin{equation}{section}

\newtheorem{thm}{Theorem}[section]

\newtheorem{lem}[thm]{Lemma}
\newtheorem{prop}[thm]{Proposition}
\newtheorem{dfn}[thm]{Definition}
\theoremstyle{remark}
\newtheorem{rem}{Remark}[section]

\newcommand{\pp}[1]{\frac{\partial}{\partial {#1}}}

\title{Entropy bounds for self-shrinkers with symmetries}
\author{John Man Shun Ma, Ali Muhammad}
\date{\today}

\begin{document}

\thanks{The authors were supported by the Danish National Research Foundation  CPH-GEOTOP-DNRF151.}


\begin{abstract}
In this work we derive explicit entropy bounds for two classes of closed self-shrinkers: the class of embedded closed self-shrinkers recently constructed in \cite{riedler} using isoparametric foliations of spheres, and the class of compact non-spherical immersed rotationally symmetric self-shrinkers. These bounds generalize the entropy bounds found in \cite{MMM} on the space of complete embedded rotationally symmetric self-shrinkers.
\end{abstract}

\maketitle

\section{Introduction}
A smooth hypersurface $\Sigma$ in $\mathbb R^{n+1}$ is called a {\sl self-shrinker} if 
\begin{equation}
H_{\Sigma}(x) = \frac{1}{2} \langle x, \nu\rangle, \ \ \text{ for all } x\in \Sigma.
\end{equation}
Here $\nu$ is the unit normal vector field on $\Sigma$, $H_\Sigma = \operatorname{div} \nu$ is the mean curvature of $\Sigma$ and $x$ is the position vector. Self-shrinkers are central objects of study in the mean curvature flow (MCF) since they serve as singularity models \cite{Huisken}, \cite{Ilmanen}. Besides the generalized cylinders, different techniques have been used to construct explicit examples of self-shrinkers \cite{AL}, \cite{ang}, \cite{KKM}, \cite{BNS}, \cite{Nguyen}, \cite{DN}, \cite{KMc}, \cite{riedler}.

In the seminal paper \cite{CM1} (see also \cite{MM}), Colding and Minicozzi defined the {\sl entropy} $\lambda(\Sigma)$ of any smooth hypersurface $\Sigma$. The entropy is a scaling and translation invariant quantity. More importantly, if $\{ \Sigma_t\}$ is a solution to the MCF, then $t\mapsto \lambda (\Sigma_t)$ is non-increasing and is constant if and only if (up to rescaling and space-time translation) $\Sigma_t = \sqrt{-t} \Sigma$ for $t \in (-\infty, 0)$, where $\Sigma$ is a self-shrinker. 

The entropy of the generalized cylinders were computed in \cite{Stone}, and the entropy of the Angenent doughnuts were numerically computed in \cite{yakov}, \cite{HN}. In general, upper bounds for the entropy are of great interest, since they normally imply surprising topological and geometrical consequences \cite{AlexGT}, \cite{BW}, \cite{CHH}, \cite{MMM}. 

In \cite{MMM}, the authors and M{\o}ller proved an entropy upper bound for complete embedded rotationally symmetric self-shrinkers in $\mathbb R^{n+1}$. Using that, several compactness and finiteness results for the space of such self-shrinkers were obtained. The goal of this paper is to generalize the entropy upper bounds obtained in \cite[Theorem 1.1]{MMM} in two directions. 

Using the theory of isoparametric foliations of spheres, Riedler constructed in \cite{riedler} new examples of closed embedded self-shrinkers in $\mathbb R^{n+1}$ with topology $\mathbb S^1 \times M$, where $M \subset \mathbb S^{n}$ is diffeomorphic to an isoparametric hypersurface of $\mathbb{S}^n$ for which the multiplicities of the principal curvatures agree. This is done by applying a reduction theorem from \cite{PT} to reduce the self-shrinker equation to a geodesic equation in an open subset of $\mathbb{R}^2$. Indeed, the construction generalizes the classical work by Angenent \cite{ang} for rotationally symmetric self-shrinkers. Our first result is the following theorem which provides entropy bounds for this class of self-shrinkers.

\begin{thm} \label{thm entropy upper bounds for f-invariant self-shrinkers}
Let $\mathscr M = \{ M_\varphi\}_{\varphi \in (0,\pi/g)}$ be an isoparametric foliation of $\mathbb S^n$ of type $(g, m,m)$, and let $N$ be a $f$-invariant embedded closed self-shrinker diffeomorphic to $\mathbb S^1 \times M$, where $M$ is diffeomorphic to a regular fiber of the foliation $\mathscr M $. Then there is a positive number $E_{g,m}$ such that $\lambda (N)\le E_{g,m}$. 
\end{thm}

We remark that the constants $E_{g,m}$ are explicit and depend only on the pair $(g, m)$. We refer the reader to section 3 for the terminology on isoparametric foliations appearing in Theorem \ref{thm entropy upper bounds for f-invariant self-shrinkers}.

Our second result is Theorem \ref{thm entropy bounds for immersed rotationally symmetric self-shrinkers} below which gives entropy bounds for immersed rotationally symmetric self-shrinkers diffeomorphic to $\mathbb S^1 \times \mathbb S^{n-1}$. Drugan and Kleene constructed in \cite{DK} infinitely many complete immersed rotationally symmetric self-shrinkers of various topological types. In the following theorem, we obtain an entropy upper bound for compact immersed non-spherical rotationally symmetric self-shrinkers in terms of the number of self-intersections of the corresponding profile curve. For each $n\ge 2$, let $E_n$ denote the entropy upper bound for complete embedded rotationally symmetric self-shrinkers in $\mathbb R^{n+1}$ obtained in \cite[equation (3.1)]{MMM}. 

\begin{thm} \label{thm entropy bounds for immersed rotationally symmetric self-shrinkers}
Let $\Sigma = \Sigma_\sigma$ be a compact immersed non-spherical rotationally symmetric self-shrinker in $\mathbb R^{n+1}$, where the profile curve $\sigma$ has $k$ self-intersection points counted with multiplicity. Then 
\begin{equation} \label{eqn entropy bound for immersed rotationally symmetric self-shrinkers}
    \lambda(\Sigma) \le (k+1) E_{n}. 
\end{equation}
\end{thm}

We refer the reader to section 4 and 5 for the terminology on rotationally symmetric self-shrinkers appearing in Theorem \ref{thm entropy bounds for immersed rotationally symmetric self-shrinkers}. Both Theorem \ref{thm entropy upper bounds for f-invariant self-shrinkers} and Theorem \ref{thm entropy bounds for immersed rotationally symmetric self-shrinkers} are proven using comparison theorems (\cite[Theorem 3.1]{MMM} and Theorem \ref{thm length bounds for immersed loop with k self intersections} respectively), which provide estimates for the length of closed geodesics $\sigma$ in a simply connected Riemannian surface with strictly positive Gaussian curvature. Theorem \ref{thm length bounds for immersed loop with k self intersections} also implies an entropy bound for a class of closed immersed $f$-invariant self-shrinkers (see remark \ref{remark immersed f-ivariant shrinkers}).



The organization of this paper is as follows. In section \ref{section intro to MCF}, we recall the definitions of self-shrinkers and the entropy of a hypersurface. In section \ref{section isoparametric foliation}, we recall the construction in \cite{riedler} and prove Theorem \ref{thm entropy upper bounds for f-invariant self-shrinkers}. In section \ref{section length bound on immersed curves}, we prove a general comparison theorem (Theorem \ref{thm length bounds for immersed loop with k self intersections}) and in the last section, we prove Theorem \ref{thm entropy bounds for immersed rotationally symmetric self-shrinkers}. 

The first author would like to thank Hojoo Lee for their discussion on immersed rotationally symmetric self-shrinkers. The authors would also like to thank Alexander Mramor for fruitful discussions. 

\section{Self-shrinkers and the entropy} \label{section intro to MCF}
Let $\Sigma$ be a properly immersed smooth hypersurface in $\mathbb R^{n+1}$. For any $x_0 \in \mathbb R^{n+1}$ and $t_0 >0$, the functional $F_{x_0, t_0}$ is defined in \cite{CM1} as 
\begin{equation}
F_{x_0, t_0} (\Sigma) = \frac{1}{(4\pi t_0)^{n/2}} \int_\Sigma e^{-\frac{\|x-x_0\|^2}{4t_0}} d\mu, 
\end{equation}
Where $d\mu$ is the volume form of $\Sigma$. It is well known (e.g. \cite[section 1]{CM2}) that the following statements are equivalent: 
\begin{itemize}
\item $\Sigma$ is a self-shrinker;
\item $\Sigma$ is a minimal hypersurface in $\mathbb R^{n+1}$ with respect to the conformal (Gau{\ss}ian) metric $g_{\text{\ss}}$, where 
\begin{equation} \label{eqn dfn of g_ang}
(g_{\text{\ss}})_{ij} = e^{-\frac{\|x\|^2}{2n}}\delta_{ij}
\end{equation}
and $\delta$ denotes the Euclidean metric on $\mathbb R^{n+1}$;
\item $\Sigma$ is a critical point of the $F$-functional $F = F_{0,1}$. 
\end{itemize}

The {\sl entropy} of a hypersurface $\Sigma$ is defined by (\cite{CM1}, \cite{MM})
\begin{equation}
\lambda (\Sigma) = \sup_{(x_0,t_0) \in \mathbb{R}^{n+1} \times \mathbb{R}_{>0}} F_{x_0, t_0} (\Sigma).
\end{equation}
We also need the following fact about the entropy of self-shrinkers \cite[section 7.2]{CM1}. 

\begin{lem} \label{lem lambda equals F 01}
Let $\Sigma$ be a properly immersed self-shrinker in $\mathbb R^{n+1}$. Then $\lambda (\Sigma) = F_{0,1} (\Sigma)$. 
\end{lem}

\section{Closed embedded self-shrinkers constructed by isoparametric foliations} \label{section isoparametric foliation}
In this section we prove Theorem \ref{thm entropy upper bounds for f-invariant self-shrinkers}. First, we recall the construction of the closed embedded self-shrinkers in \cite{riedler} using the theory of isoparametric foliations of spheres, which generalizes the classical construction of rotationally symmetric self-shrinkers by Angenent \cite{ang}. The definition of isoparametric foliations and the properties of such foliations of spheres, in particular, the structural theorems proved by M{\" u}nzner in \cite{munzner}, \cite{munzner2}, are neatly summarized in \cite[section 2]{riedler} and we closely follow their notation. For a general introduction, see for example \cite[Chapter 3]{CR}. We start by briefly recalling the necessary background on the theory. We then proceed to describe how the problem of finding self-shrinkers in the setting of \cite{riedler} is reduced to finding geodesics in a certain open subset of $\mathbb{R}^2$ equipped with a family of (incomplete) metrics $h_{g,m}=h$ with Gaussian curvature bounded from below by strictly positive numbers.

\subsection{Background on Isoparametric Foliations of $\mathbb{S}^n$} A smooth function $F: \mathbb{S}^n \to \mathbb{R}$ is called {\sl isoparametric} if there exist smooth functions $\phi_1, \phi_2 : F(\mathbb{S}^n) \to  \mathbb{R}$ such that 
\begin{align*}
    \|\text{grad} F\|^2 = \phi_1 \circ F, \quad \Delta F = \phi_2 \circ F,
\end{align*}
where $\text{grad} F$ and $\Delta F$ denote the gradient and the Laplacian of $F$, respectively. Foliations of $\mathbb{S}^n$ that arise from a family of level sets of an isoparametric function are called {\sl isoparametric foliations} of $\mathbb{S}^n$. Such level sets are then the fibers of a given foliation. By the work of M{\" u}nzner \cite{munzner}, the isoparametric function of a given foliation is a restriction to $\mathbb{S}^n$ of a homogeneous polynomial $F: \mathbb{R}^{n+1} \to \mathbb{R}$ of degree $g$ which satisfies certain differential equations, where $g\in \{1, 2, 3, 4, 6\}$. Such a polynomial $F$ is called the {\sl Cartan-M{\" u}nzner polynomial}. The foliation has exactly two singular fibers $V_\pm = (F|_{\mathbb S^n})^{-1} (\pm 1)$, and each regular fiber is given by 
\begin{equation}\label{definition regular fiber}
    M_\varphi = \{ x\in \mathbb S^n | \operatorname{dist}_0(x, V_-) = \varphi\},
\end{equation}
where $\varphi \in (0,\pi/g)$ and $\operatorname{dist}_0$ is the distance function on the round $\mathbb{S}^n$. On each $M_\varphi$, there are $g$ (the degree of the Cartan-M{\" u}nzner polynomial $F$) distinct principal curvatures, and they each assume a constant value on a given regular fiber. Furthermore, there are at most two distinct multiplicities for the principal curvatures which we denote by $m_1$ and $m_2$, and we have the following relation $n-1 = \frac{g}{2}(m_1+m_2)$. Also, for each $\varphi \in (0, \pi/g)$ we have:
\begin{equation}\label{regular fiber volume}
    \operatorname{Vol}_0 (M_\varphi) = c \sin \left(\frac{g}{2} \varphi \right)^{m_1} \cos \left(\frac{g}{2} \varphi \right)^{m_2},
\end{equation}
where $c$ is some constant independent of $\varphi$ and $\operatorname{Vol}_0 (M_\varphi)$ is the volume of $M_\varphi$ with respect to the round metric on $\mathbb{S}^n$.

We shall call an isoparametric foliation $\mathscr M = \{ M_\varphi\}$ of type $(g,m_1,m_2)$ if the fibers have $g$ distinct principal curvatures with corresponding multiplicities $m_1$ and $m_2$. In the following we shall restrict to the case where the multiplicities agree, i.e. $m_1=m_2$ and hence consider foliations of type $(g,m,m)$. In this case we have $n=mg+1$ and we shall use $n$ and $m$ interchangeably through this relation.


In the next lemma we determine the number $c$ in \eqref{regular fiber volume} in terms of $g$ and $m$. We start by recalling the following useful result: let $\pi : (E, g_E)\to (B, g_B)$ be a surjective Riemannian submersion between two Riemannian manifolds with compact fibers. 
Let $v : B \to \mathbb R$ be $v(x) = \operatorname{Vol}_{g_E} (\pi^{-1}(x))$. By the co-area formula \cite[Theorem 2.1]{Nicolaescu}, for all submanifolds $M$ of $B$ we have

\begin{equation} \label{eqn vol of preimage of M}
\operatorname{Vol}_{g_E} (\widetilde M) = \int_M v(x) d\mu_M (x),
\end{equation}
where $\widetilde M = \pi^{-1}(M)$ and $d\mu_M$ is the volume form of $M$ in $(B, g_B)$. 

\begin{lem} \label{lem volume of M_varphi}
Given an isoparametric foliation $\{ M_\varphi\}$ of $\mathbb S^n$ of type $(g, m, m)$, we have 
\begin{equation} \label{eqn volume of regular fiber of an isoparametric foliation}
\operatorname{Vol}_0(M_\varphi) = \frac{g\omega_n }{s(m)} \sin^m (g\varphi)
\end{equation}
for every regular fiber $M_\varphi$, $\varphi \in (0, \pi/g)$. Here $n = mg+1$,  $\omega_n$ is the volume of $\mathbb S^n$ and 
\begin{equation} \label{eqn dfn of s(m)}
s(m) = \int_0^\pi \sin ^m t \; \mathrm{d}t .
\end{equation}
\end{lem}

\begin{proof}
Since $m_1 = m_2 = m$ we have from \eqref{regular fiber volume} 
\begin{equation} \label{eqn vol of M_varphi in terms of c}
\operatorname{Vol}_0 (M_\varphi) = \frac{c}{2^m} \sin^m (g\varphi), 
\end{equation}
hence it suffices to find $c$. By the relation \eqref{definition regular fiber} we know that the mapping
\begin{align*}
    \mathbb S^n\setminus (V_- \cup V_+)&\to \left(0, \frac{\pi}{g}\right), \\
    x\in M_\varphi & \mapsto \varphi,
\end{align*}
is a Riemannian submersion, with respect to the standard metric on $\mathbb S^n$ and the Euclidean metric on the interval $(0,\pi/g)$. Using (\ref{eqn vol of preimage of M}) and (\ref{eqn vol of M_varphi in terms of c}), 
\begin{align*}
\omega_n := \operatorname{Vol}_0 (\mathbb S^n) = \int_0^{\pi/g} \operatorname{Vol}_0 (M_\varphi) d\varphi = \frac{c}{g2^m} \int_0^\pi \sin^m t \; \mathrm{d}t. 
\end{align*}
Hence $c = g2^m \omega_n s(m)^{-1}$ and Lemma \ref{lem volume of M_varphi} is proved.
\end{proof}
For simplicity, from now on we write 
\begin{equation} \label{eqn dfn of c_g,m}
c_{g, m} = \frac{g \omega_n}{s(m)}. 
\end{equation}
\subsection{The Reduction Theory; Proof of Theorem \ref{thm entropy upper bounds for f-invariant self-shrinkers}}
Equip $\mathbb R^{n+1} \backslash \mathbb R_{\ge 0} \cdot (V_+ \cup V_-)$ with the metric $g_{\text{\ss}}$ defined in (\ref{eqn dfn of g_ang}), and let 
\begin{align}\label{eqn definition of g_subm}
g_{\text{S}} = e^{-\frac{r^2}{2n}} (dr^2 + r^2 d\varphi^2)
\end{align}
be a metric on $(0, \infty) \times (0, \pi/g)$.
Define the map
\begin{align}
f : \mathbb R^{n+1} \backslash \mathbb R_{\ge 0} \cdot (V_+ \cup V_-) &\to (0,\infty) \times \left( 0, \frac{\pi}{g}\right),\\
\nonumber x &\mapsto \left( \| x\|, \frac{\arccos (F(x/\|x\|))}{g}\right). 
\end{align}
Here $F$ is the corresponding Cartan-M{\"u}nzner polynomial and $V_\pm = (F|_{\mathbb S^n})^{-1} (\pm 1)$ are the singular fibers of the foliation. It can be shown that $f$ is a Riemannian submersion. A set $N \subset \mathbb{R}^{n+1}$ is called $f$-{\sl invariant} if there exists a set $C \subset (0,\infty) \times (0, \pi/g)$ such that $N = f^{-1}(C)$. 

Using \cite[Theorem 4]{PT}, it is proved in \cite[Proposition 2.4]{riedler} that an $f$-invariant hypersurface $N \subset \mathbb{R}^{n+1}$ is a closed self-shrinker if and only if $C := f(N)$ is a closed geodesic in $(0,\infty) \times (0,\pi/g)$ with respect to the metric $h$ defined by
\begin{equation} \label{eqn definition of h}
h(r,\phi) := \operatorname{Vol}_{g_{\text{\ss}}} (f^{-1} (r, \varphi))^2 g_{\text{S}}(r,\phi).
\end{equation}
Note that for any $(r, \varphi) \in (0,\infty) \times (0,\pi/g)$, we have by Lemma \ref{lem volume of M_varphi}
\begin{align*}
\operatorname{Vol}_{g_{\text{\ss}}} (f^{-1} (r, \varphi)) & = \operatorname{Vol}_{g_{\text{\ss}}} (r M_\varphi) \\
&= e^{-\frac{(n-1)r^2}{4n}} \operatorname{Vol}_0 (r M_\varphi) \\
&= r^{n-1}e^{-\frac{(n-1)r^2}{4n}} \operatorname{Vol}_0 (M_\varphi) \\
&= c_{g,m} r^{n-1}e^{-\frac{(n-1)r^2}{4n}} \sin^m (g\varphi).
\end{align*}
Using \eqref{eqn definition of g_subm} we hence obtain the following expression for the metric $h$\footnote{In equation (4) of \cite{riedler}, the metric has the term $e^{-r^2}$ instead of $e^{-r^2/2}$ since they use a different scaling convention for self-shrinkers.}
\begin{equation}\label{eqn definition of h expression}
    h(r,\phi) =  c^2_{g,m}r^{2mg} e^{-r^2/2} \sin \left( g \varphi\right)^{2m}  (dr^2 + r^2 d\varphi^2).
\end{equation}
By a direct calculation, the Gaussian curvature $K = K_{g,m}$ of $h$ is given by
\begin{equation*}
K = \frac{ e^{r^2/2}}{c^2_{g,m} r^{2n} \sin \left( g \varphi\right)^{2m}} \left( r^2 + \frac{(n-1)g}{ \sin ^2 (g\varphi)} \right).
\end{equation*}
It is clear that $K$ is strictly positive and that $K \to +\infty$ as $(r, \varphi)$ tend to the boundary of $(0,\infty) \times (0,\pi/g)$. By simple calculus, the Gaussian curvature $K$ is bounded from below by $\kappa = \kappa_{g,m}>0$, where 
\begin{equation} \label{eqn expression for kappa_gm}
\kappa _{g,m}= \frac{ e^{\frac{y_{g,m}}{2}}}{c_{g,m}^2 y_{g,m}^n}\left( y_{g,m} + (n-1)g \right)
\end{equation}
and 
\begin{equation}
y_{g,m} =\frac{(n-1)(2-g) + \sqrt{(g-2)^2(n-1)^2 + 8ng(n-1)}}{2}.
\end{equation}
Moreover, by \eqref{eqn definition of h}, (\ref{eqn vol of preimage of M}) and \eqref{eqn dfn of g_ang},
the length $L_h (C)$ of the curve $C$ with respect to the metric $h$ is given by
\begin{align} \label{eqn length of C equal volume}
L_h (C) &= \int_C \operatorname{Vol}_{g_{\text{\ss}}} (f^{-1} (r, \varphi)) d\mu_C(r,\phi) \nonumber\\
&= \operatorname{Vol}_{g_{\text{\ss}}} (f^{-1} (C)) \nonumber \\
&= \operatorname{Vol}_{g_{\text{\ss}}} (N) \nonumber\\
&= \int_N e^{-\frac{\|x\|^2}{4}} d\mu_N. 
\end{align} 
Where $d\mu_C$ is the volume form of $C$ in $((0, \infty) \times (0, \pi/g), g_{\text{S}})$. 

\begin{proof}[Proof of Theorem \ref{thm entropy upper bounds for f-invariant self-shrinkers}]
For each pair $(g, m)$ let 
\begin{equation} \label{eqn expression for E_gm}
    E_{g,m} = \frac{2\pi }{(4\pi)^{n/2} \sqrt{\kappa_{g,m}}}. 
\end{equation} 
Let $N$ be a $f$-invariant embedded self-shrinker in $\mathbb R^{n+1}$. Then $C:= f(N)$ is the image of an embedded geodesic in $(0, \infty) \times (0, \pi/g)$ with metric $h$ given in (\ref{eqn definition of h}). Note that $N$ is diffeomorphic to $\mathbb S^1 \times M$ if and only if $f(N)$ is an embedded closed geodesic. Hence we can apply \cite[Theorem 3.1]{MMM} to conclude that 
$$L_h (C) \le \frac{2\pi }{\sqrt{\kappa_{g,m}}}.$$
Together with (\ref{eqn length of C equal volume}) and Lemma \ref{lem lambda equals F 01}, one obtains the entropy bound $\lambda (N)\le E_{g,m}$. 
\end{proof}

\begin{rem}
When $g=1$, the isoparametric foliation of $\mathbb S^n$ is given (up to congruence) by a family of $(n-1)$-spheres 
\begin{equation}
M_\varphi = \{ \cos\varphi \} \times \sin\varphi \ \mathbb S^{n-1}, \ \ \varphi \in (0,\pi) 
\end{equation}
and $f$-invariant hypersurfaces are percisely the rotationally symmetric hypersurfaces. In \cite{MMM}, the authors and M{\o}ller derived entropy bounds for complete embedded rotationally symmetric self-shrinkers. We remark that the entropy bounds found in \cite{MMM} are a special case of the ones obtained here: when $g=1$, then $m=n-1$ and we obtain from (\ref{eqn expression for E_gm}) that 
\begin{equation*}
E_{1, n-1} = \frac{2\pi}{(4\pi)^{n/2} \sqrt{\kappa_{1,n-1}}}. 
\end{equation*}
From (\ref{eqn expression for kappa_gm}) and $\kappa_n$ found in \cite[section 3]{MMM}, we have 
\begin{equation*}
\sqrt{ \kappa_{1,n-1}} = \frac{1}{c_{1,n-1}} \sqrt{\kappa_n}. 
\end{equation*}
Together with the identity $\omega_n = s(n-1) \omega_{n-1}$, (\ref{eqn dfn of c_g,m}) and the expression of $E_n$ in \cite[equation (3.1)]{MMM}, we have $E_{1,n-1} = E_n$ for all $n\ge 2$. 
\end{rem}

\section{Length upper bound on immersed closed geodesics in Riemannian surfaces with positive curvature} \label{section length bound on immersed curves}
In this section we prove Theorem \ref{thm length bounds for immersed loop with k self intersections}, which gives an upper bound on the length of closed immersed geodesics in a simply connected Riemannian surface with Gaussian curvature bounded below by a strictly positive constant $\kappa$. We start with a short outline of the section.

An immersed closed geodesic $\sigma$ encloses a number of domains. In Proposition \ref{prop gamma bounds k+1 domains} we determine the number of these domains in terms of the number of self-intersection points of $\sigma$. The proof of Theorem \ref{thm length bounds for immersed loop with k self intersections} then reduces to proving an estimate of the length of the boundary of each domain. This is done by showing that the domains are complete length spaces of curvature $\geq \kappa$, for which one can prove an analogous result as in \cite[Theorem 3.1]{MMM}. The necessary technical details are provided in Lemma \ref{shorest geo in Omega} and Proposition \ref{prop Omega domega has curvature ge kappa}. \\



Let $\mathbb V $ be a Riemannian $2$-manifold homeomorphic to $\mathbb R^2$. 

\begin{dfn}
Let $\gamma: \mathbb S^1 \to \mathbb V $ be an immersed closed curve. Let $\operatorname{Image}(\gamma)$ denote the image of $\gamma$. A point $p$ in $\operatorname{Image}(\gamma)$ is called a self-intersection point if $\gamma^{-1}(p)$ has more than one element. A self-intersection point $p$ is transverse if for all $t, s\in \gamma^{-1}(p)$, the pair $\gamma'(s), \gamma'(t)$ are not parallel to each other. Let $S$ be the set of all self-intersection points of $\gamma$ and let $k\in \mathbb N$. We say that $\gamma$ has $k$ self-intersection points counted with multiplicity if 
\begin{equation} \label{eqn dfn of k self intersection counted with multiplicity}
\sum_{p\in S} (|\gamma^{-1}(p)|-1) = k.
\end{equation}
\end{dfn}

\begin{dfn}
Let $\gamma$ be an immersed closed curve in $\mathbb V$ with only transverse self-intersections. A pre-compact connected component of $\mathbb V \setminus \operatorname{Image}(\gamma)$ is called a domain enclosed by $\gamma$.
\end{dfn}

\begin{prop}  \label{prop interior angles < pi}
Let $\gamma : \mathbb S^1 \to \mathbb V$ be an immersed closed curve in $\mathbb V$ with only transverse self-intersections. Then the boundary $\partial U$ of each domain $U$ enclosed by $\gamma$ is piece-wise smooth and at each corner of $\partial U$, the interior angle is less than $\pi$. 
\end{prop}

\begin{proof}
It is clear that the boundary of $U$ is piece-wise smooth. At each corner, the interior angle $\theta$ is not $\pi$ since $\gamma$ has transverse self-intersections. If $\theta >\pi$, then there is a sub-arc $\beta$ of $\gamma$ which lies in the interior of $U$ and $\beta$ is not part of $\partial U$. Let $\beta$ be the boundary of a domain $\tilde U$ enclosed by $\gamma$. But since $\beta$ lies in the interior of $U$, we must have $U = \tilde U$, which is a contradiction that $\beta$ is not part of $\partial U$. Hence we have $\theta < \pi$. 
\end{proof}

\begin{prop} \label{prop gamma bounds k+1 domains}
Let $\gamma$ be an immersed closed curve in $\mathbb V$ with $k$ transverse self-intersections counted with multiplicity. Then $\gamma$ encloses exactly $k+1$ domains.
\end{prop}

\begin{proof}
Let $U_1,\cdots, U_f$ be the domains enclosed by $\gamma$. For every $j=1, \cdots, f$ the closure of the domain $U_j$ has a piece-wise smooth boundary consisting of subarcs of $\gamma$. Let $\alpha_1 , \cdots, \alpha_e$ be the collection of those sub-arcs. Let $S = \{p_1, \cdots, p_\ell\}$ be the set of self-intersection points of $\gamma$ and let $W = \overline {U_1 \cup \cdots \cup U_f}$. Since $W$ is simply connected, the Euler formula implies 
\begin{equation}
\chi(W) = 1 = \ell - e + f,
\end{equation}
where $\chi(W)$ is the Euler characteristic of $W$. Let $s_i = |\gamma^{-1}(p_i)|$. By definition there are $2s_i$ sub-arcs with vertices $p_i$ (in the case where there is a loop, i.e., a sub-arc which starts and ends at $p_i$, that sub-arc would be counted twice). Since each sub-arc has exactly two vertices (also counted with multiplicity), we have
$$e = \sum_{i=1}^\ell s_i$$
and hence
\begin{equation}
f = 1 + e - \ell = 1 + \sum_{i=1}^\ell (s_i-1) = 1+k
\end{equation}
by (\ref{eqn dfn of k self intersection counted with multiplicity}). This finishes the proof of the proposition. 
\end{proof}

Let $h$ be a Riemannian metric on $\mathbb V$ such that the Gaussian curvature $K$ satisfies $K \geq \kappa > 0$ for some positive constant $\kappa$, and let $\sigma$ be an immersed closed geodesic in $(\mathbb V, h)$. By the uniqueness theorem of solutions to ODE, any self-intersection point of $\sigma$ must be transverse. We shall now describe the type of domains $U \subset \mathbb V$ that can be realized as the domains enclosed by $\sigma$. 

Let $U$ be a precompact, connected open subset in $(\mathbb V, h)$ and let $\Omega:=\overline U \subset \mathbb V$. Assume further that the boundary $\partial \Omega$ is parameterized by a curve $\beta : \mathbb S^1 \to (\mathbb V, h)$ consisting of piece-wise smooth geodesic arcs, and which furthermore satisfies the following properties: identifying $\mathbb S^1$ with $[0,d]$, with $0$ and $d$ identified, then there is a partition
$$\mathscr P =\{0 = t_0 < t_1 <\cdots < t_a = d (=0)\}$$
so that 
\begin{equation} \label{conditions on beta}
    \begin{cases} 
    \beta|_{\mathbb S^1 \setminus \mathscr P} \text{ is injective,} \\
    \text{for each }j=0, \cdots, a-1, \ \beta|_{[t_j, t_{j+1}]} \text{ is a geodesic in }(\mathbb V, h), \text{ and}\\
    \text{for each }j=0, \cdots, a-1 \text{ the interior angle at }q_j := \beta(t_j) \text{ is less than }\pi.
    \end{cases}
\end{equation} 

We remark that if $\beta$ is injective, then each vertex $q_j$ is joined by only two geodesic arcs $\beta|_{[t_{j-1}, t_j]}$, $\beta|_{[t_j, t_{j+1}]}$. 

Assume first that $\partial \Omega$ is parametrized by an injective curve $\beta$ which satisfies (\ref{conditions on beta}). Using Lemma \ref{shorest geo in Omega} below, we will define as in \cite[Section 3]{MMM} a metric structure $d^\Omega$ on $\Omega$ using the Riemannian metric $h$ and show that $(\Omega, d^\Omega)$ is a complete length space with curvature $\ge \kappa$, where $\kappa$ is the strictly positive lower bound of the Gaussian curvature of $(\mathbb V, h)$. See \cite{BBI} or \cite[Section 3]{MMM} for the terminology and notation needed from metric geometry. 

For each $p, q\in \Omega$, let 
$$ d^\Omega(p, q) = \inf_{\gamma} L_h(\gamma),$$
where the infimum is taken among piece-wise $C^1$ curves $\gamma : [0,1]\to \Omega$ with $\gamma(0) = p$, $\gamma(1) = q$ and 
\begin{equation}
    L_h(\gamma) =\int_0^1 \sqrt{h(\gamma'(t), \gamma'(t))} dt
\end{equation} 
is the length of $\gamma$ in $(\mathbb V, h)$. The following lemma is proved in \cite[Lemma 3.3]{MMM} when $\partial \Omega$ has no corners. 

\begin{lem} \label{shorest geo in Omega}
Let $\Omega$ be as above, i.e. a compact connected set in $(\mathbb V, h)$ so that the boundary $\partial \Omega$ is parametrized by an injective curve $\beta$ which satisfies (\ref{conditions on beta}). Let $p, q\in \Omega$. Then there is a simple geodesic $\gamma$ in $\Omega$ joining $p$ and $q$ which is shortest among all piece-wise $C^1$ curves in $\Omega$ joining $p$ and $q$. Moreover, 
\begin{itemize}
\item [(i)] $\gamma (t) \notin \{ q_1, \cdots, q_a\}$ for all $t\in (0,1)$.
\item [(ii)] if one of $p, q$ is in the interior of $\Omega$, then $\gamma$ also lies in the interior of $\Omega$ (except possibly at the other end point), 
\item [(iii)] if both $p, q$ are in $\operatorname{Image}(\beta) = \partial \Omega$, then either $\gamma$ lies completely in the image of $\beta|_{[t_j, t_{j+1}]}$ for some $j=0, \cdots, a-1$, or the interior of $\gamma$ lies inside the interior of $\Omega$.
\end{itemize}
\end{lem}

\begin{proof}
The Lemma follows from the same arguments used in the proof of \cite[Lemma 3.3]{MMM}. We only sketch the minor modifications here. 

Let $p, q\in \Omega$, $p\neq q$. As in the proof of \cite[Lemma 3.3]{MMM}, using a minimizing sequence we construct a Lipschitz length minimizing curve $\gamma : [0,1] \to \Omega$ with $\gamma(0)=p$, $\gamma(1) = q$, and we may assume that $\gamma ^{-1} (q) = \{1\}$, $\gamma^{-1}(p) = \{0\}$. 

First we show (i) by contradiction: if not, let $s_0 \in (0,1)$ be the first time such that $\gamma(s_0) = q_j$ for some $j$. Let $B$ be the closed geodesic ball of radius $\epsilon$ in $(\mathbb V, h)$ centered at $q_j$. For $\epsilon$ small enough, $p, q$ are not in $B$. Hence there are $s_1, s_2 \in (0,1)$ so that $s_1 < s_0< s_2$ and $\gamma (s_i) \in \partial B$ for $i=1,2$. Note that since $\beta|_{[t_{j-1}, t_j]}$ and $\beta_{[t_j, t_{j+1}]}$ are both geodesics, under geodesic polar coordinates centered at $q_j$, the set $B\cap \Omega$ is of the form
\begin{equation} \label{eqn U_x in polar coordinates}
    \{ (r, \theta) : 0\le r \le \epsilon, \theta_1 \le \theta\le \theta_2\}
\end{equation}
with $\theta_2 - \theta_1 <\pi$ by the third assumption in (\ref{conditions on beta}). Assuming $\epsilon$ is small enough so that $g$ is nearly Euclidean in $U$. Together with $\theta_2-\theta_1<\pi$, the (Euclidean) straight line $\ell$ joining $\gamma(s_1)$, $\gamma(s_2)$ has length
\begin{align*}
    L_h(\ell) < d_h (\gamma(s_1), q_j) + d_h (q_j , \gamma (s_2)) \le L_h (\gamma|_{[s_1, s_2]}). 
\end{align*}
Thus $\gamma$ cannot be length minimizing in $\Omega$ and hence (i) is proved. Part (ii) and (iii) of the Lemma can be proven similarly as in the proof of \cite[Lemma 3.3]{MMM} and are skipped. 
\end{proof}

From Lemma \ref{shorest geo in Omega} it follows that $(\Omega, d^\Omega)$ is a complete length space.

\begin{prop} \label{prop Omega domega has curvature ge kappa}
Let $\Omega$ be as in Lemma \ref{shorest geo in Omega}. Then the complete length space $(\Omega, d^\Omega)$ has curvature $\ge \kappa$, where $\kappa$ is the lower bound of the Gaussian curvature of $(\mathcal S, h)$. 
\end{prop}

\begin{proof}
Let $x\in \Omega$ and let $V_x$ be a geodesically convex neighborhood in $(\mathbb V, h)$ centered at $x$. When $x$ is in the interior of $\Omega$, we choose $V_x\subset \Omega$. Let $U_x = \Omega \cap V_x$. For any $p, q\in U_x$, let $\gamma$ be the unique shortest geodesic in $V_x$ joining $p$ and $q$. Arguing as in the proof of \cite[Proposition 3.6]{MMM}, it suffices to show that $\gamma$ lies in $U_x$. The argument there works in our case when $x\notin \{ q_1, \cdots, q_a\}$. For the case $x\in \{q_1, \cdots, q_a\}$, one may assume that $V_x$ is of the form (\ref{eqn U_x in polar coordinates}) for some $r_0 >0$ and $\theta_2 - \theta_1 <\pi$. Arguing as in the proof of \cite[Proposition 3.6]{MMM}, one shows that $\gamma$ lies in $H_{\theta_0}$ for all $\theta_0 \in [\theta_2, \theta_1 + \pi] $, where in geodesic polar coordinates  
$$H_{\theta_0} =\{ (r, \theta) : 0\le r\le r_0, \theta_0-\pi \le \theta\le \theta_0 \}.$$
Since 
$$U_x = \bigcap _{\theta_0 \in (\theta_2, \theta_1 + \pi)} H_{\theta_0},$$
one concludes that $\gamma$ lies in $U_x$. 
\end{proof}

Using exactly the same proof, by approximating the boundary of $\Omega$ by geodesic polygons, one has the following generalization of \cite[Theorem 3.1]{MMM}. 

\begin{thm} \label{thm length bound on boundary of Omega}
Let $\Omega$ be as in Lemma \ref{shorest geo in Omega}. Then the boundary $\partial \Omega$ satisfies 
\begin{equation}
|\partial \Omega|:= L_h(\beta) \le \frac{2\pi }{\sqrt{\kappa}}.  
\end{equation}
\end{thm}

Now we are ready to prove the main theorem in this section. 

\begin{thm} \label{thm length bounds for immersed loop with k self intersections}
Let $\sigma$ be an immersed geodesic loop in $(\mathbb V, h)$ with $k$ self-intersection points counted with multiplicity. Then the length $L_h(\sigma)$ of $\sigma$ satisfies
\begin{equation}
L_h(\sigma) \le (k+1) \frac{2\pi }{\sqrt{\kappa}}. 
\end{equation}
\end{thm}

\begin{proof}
By Proposition \ref{prop gamma bounds k+1 domains}, the immersed geodesic loop $\sigma$ encloses $k+1$ domains, denote the closure of these domains by $\Omega_1, \cdots, \Omega_{k+1}$. Fix each $i=1, \cdots, k+1$ and write $\Omega = \Omega_i$ for simplicity. The boundary $\partial \Omega$ of $\Omega$ is parametrized by a closed curve $\beta : \mathbb S^1 \to (\mathbb V, h)$ consisting of piece-wise smooth geodesic arcs which satisfies (\ref{conditions on beta}), but which may not be injective. Let $q$ be any vertex of $\partial \Omega$ (i.e. $q=\beta(t_j)$ for some $j=0, \cdots, a$). Let $s= |\beta^{-1}(q)|$. For any small $\epsilon >0$, let $B= B_\epsilon(q)$ be the closed geodesic ball in $(\mathbb V, h)$ centered at $q$ with radius $\epsilon$. Then $\beta$ intersects $\partial B$ at $2s$ points, given by $\exp_q(\epsilon e^{i \theta_0}), \cdots, \exp_q(\epsilon e^{i \theta_{2s-1}})$ in geodesic polar coordinates $(r, \theta)$ centered at $q$. Moreover, we have 
\begin{equation*}
\Omega \cap B = \{ \exp_q(r e^{i\theta}):  0\le r\le \epsilon, \ \  \theta_{2m}\le \theta \le \theta_{2m+1}, m=0, \cdots, s-1\}. 
\end{equation*}
For each $m$, let $\beta_m$ be the shortest geodesic in $B$ joining $\exp_q (\epsilon e^{i\theta_{2m}})$ and $\exp_q (\epsilon e^{i\theta_{2m+1}})$. Arguing as in the proof of Proposition \ref{prop Omega domega has curvature ge kappa}, $\beta_m$ also lies inside $\Omega$. Let $\Delta_m$ be the geodesic triangle in $\Omega$ with vertices $q$, $\exp_q (\epsilon e^{i\theta_{2m}})$ and $\exp_q(\epsilon e^{i\theta_{2m+1}})$. Let $D_\epsilon$ be the collection of all such triangles constructed for all vertices of $\partial \Omega$. Let $\Omega_\epsilon = \Omega \setminus D_\epsilon$. Note that the interior of $\Omega$ is homeomorphic to the interior of $\Omega_\epsilon$. Hence $\Omega_\epsilon$ is connected. Moreover, the boundary $\partial \Omega_\epsilon$ is parametrized by an injective curve consisting of piece-wise smooth geodesic arcs which satisfies (\ref{conditions on beta}). Thus we can apply Theorem \ref{thm length bound on boundary of Omega} to conclude 
\begin{equation*}
|\partial \Omega_\epsilon| \le \frac{2\pi}{\sqrt{\kappa}}. 
\end{equation*}
Taking $\epsilon \to 0$, we obtain 
\begin{equation*}
|\partial \Omega| \le \frac{2\pi}{\sqrt{\kappa}}
\end{equation*}
and hence 
\begin{equation*}
L_h(\sigma) \le \sum_{i=1}^{k+1} |\partial \Omega_i| \le (k+1) \frac{2\pi}{\sqrt{\kappa}}. 
\end{equation*}
\end{proof}

\section{Entropy bounds for immersed non-spherical closed self-shrinkers}
We recall that for any immersed curve $\sigma : I \to \mathbb H$, where $\mathbb H$ is the upper half space, the immersed rotationally symmetric hypersurface $\Sigma_\sigma$ in $\mathbb R^{n+1}$ with profile curve $\sigma (s)= (x(s), r(s))$ is given by 
\begin{equation*}
\Sigma_\sigma := \{ (x(s), r(s) \omega) : s\in I, \ \omega\in \mathbb S^n\}. 
\end{equation*}
Now we prove Theorem \ref{thm entropy bounds for immersed rotationally symmetric self-shrinkers}.

\begin{proof}[Proof of Theorem \ref{thm entropy bounds for immersed rotationally symmetric self-shrinkers}] By \cite[Proposition 2.3]{MMM}, $\Sigma_\sigma$ is a self-shrinker if and only if $\sigma$ is a geodesic in $(\mathbb H, g_A)$, where $g_A$ is given by 
\begin{align}\label{dfn of ang metric}
g_A = r^{2(n-1)} e^{-\frac{x^2 + r^2}{2}} (dx^2 + dr^2).
\end{align} 
Since $\Sigma_\sigma$ is compact, either $I = \mathbb S^1$ and $\Sigma_\sigma$ is an immersion from $\mathbb S^1 \times \mathbb S^{n-1}$. Or $I = (a, b)$, $\sigma (a^+), \sigma(b^-)$ lie in $\partial \mathbb H$ and $\Sigma_\sigma$ is an immersion from $\mathbb S^n$. Since $\Sigma_\sigma$ is non-spherical, the latter case is excluded and thus $\sigma$ is a closed immersed geodesic. By Theorem \ref{thm length bounds for immersed loop with k self intersections}, the length of $\sigma$ satisfies 
\begin{equation*}
L(\sigma) \le (k+1) \frac{2\pi}{\sqrt {\kappa_n}},
\end{equation*}
where $\kappa_n$ is the lower bound of the Gaussian curvature of $g_A$ computed in \cite[section 3]{MMM}. Together with
\begin{align} \label{relation between entropy and length}
 \lambda (\Sigma_\sigma) = (4\pi)^{-n/2} \omega_{n-1} L_A(\sigma)
\end{align}
and the numbers $E_n$ from \cite[equation (3.1)]{MMM}, given by
\begin{align*}
    E_n = \frac{2\pi \omega_{n-1}}{(4\pi)^{n/2} \sqrt{\kappa_n}}
\end{align*}
we obtain (\ref{eqn entropy bound for immersed rotationally symmetric self-shrinkers}). 
\end{proof}

\begin{rem}\label{remark immersed f-ivariant shrinkers}
    As in the proof of Theorem \ref{thm entropy bounds for immersed rotationally symmetric self-shrinkers}, one can use Theorem \ref{thm length bounds for immersed loop with k self intersections} to prove an entropy upper bound for the class of $f$-invariant closed self-shrinkers $N$ such that $f(N)$ is a closed immersed geodesic with $k$ self-intersections in $(0, \infty) \times (0, \pi/g)$ with respect to the metric $h$ given in \eqref{eqn definition of h expression}. However, as of now no such examples are known to the authors' knowledge.
\end{rem}

\bibliographystyle{abbrv}

\vspace{0.75cm}

    \noindent John Man Shun Ma \\
    Department of Mathematics, Southern University of Science and Technology \\
    No 1088, xueyuan Rd., Xili, Nanshan District, Shenzhen, Guangdong, China 518055 \\
    email: hunm@sustech.edu.cn

\bigskip

    \noindent Ali Muhammad \\
    Department of Mathematical Sciences, University of Copenhagen \\
    Universitetsparken 5, DK-2100 Copenhagen \O, Denmark\\
    email: ajhm@math.ku.dk

\end{document}